\documentclass[12pt]{article}
\usepackage{amssymb, amsmath, latexsym, a4}
\usepackage[latin1]{inputenc}
\usepackage{graphicx}
\usepackage[all]{xy}
\usepackage[usenames]{color}

\newenvironment{proof}[1][Proof]{\noindent\textbf{#1.} }{\ \rule{0.5em}{0.5em}}

\newtheorem{De}{Definition}[section]
\newtheorem{Th}[De]{Theorem}
\newtheorem{Pro}[De]{Proposition}
\newtheorem{Le}[De]{Lemma}
\newtheorem{Co}[De]{Corollary}
\newtheorem{Rem}[De]{Remark}
\newtheorem{Ex}[De]{Example}

\newcommand{\A}{\ensuremath{\mathcal{A}}}

\newcommand{\Lieg}{\ensuremath{\mathfrak{g}}}

\newcommand{\LieK}{\ensuremath{\mathcal{K}}}

\newcommand{\LieH}{\ensuremath{\mathcal{H}}}

\newcommand{\LieI}{\ensuremath{\mathfrak{I}}}

\newbox\pullbackbox
\setbox\pullbackbox=\hbox{\xy 0;<1mm,0mm>: \POS(4,0)\ar@{-} (0,0)
\ar@{-} (4,4)
\endxy}

\newdir{>>}{{}*!/3.5pt/:(1,-.2)@^{>}*!/3.5pt/:(1,+.2)@_{>}*!/7pt/:(1,-.2)@^{>}*!/7pt/:(1,+.2)@_{>}}
\newdir{ >>}{{}*!/8pt/@{|}*!/3.5pt/:(1,-.2)@^{>}*!/3.5pt/:(1,+.2)@_{>}}
\newdir{ |>}{{}*!/-3.5pt/@{|}*!/-8pt/:(1,-.2)@^{>}*!/-8pt/:(1,+.2)@_{>}}
\newdir{ >}{{}*!/-8pt/@{>}}
\newdir{>}{{}*:(1,-.2)@^{>}*:(1,+.2)@_{>}}
\newdir{<}{{}*:(1,+.2)@^{<}*:(1,-.2)@_{<}}

\newcommand{\g}{\frak g}

\begin{document}
\title{\bf Descending Chain Conditions on
Leibniz Algebras.}
\author{Calvin Tcheka$^{1},$  Guy R. Biyogmam$^{2},$\\ Bell Bogmis N. G$^{3}$,
Batkam Mbatchou V. Jacky III$^{4}$.}
\maketitle

\bigskip
\centerline{$^{1}$ Department of Mathematics, Faculty of
science-University of Dschang} \centerline{Campus Box (237)67
Dschang, Cameroon} \centerline{ {E-mail address}:
calvin.tcheka@univ-dschang.org}
\bigskip
\centerline{\small $^{2}$ Department of Mathematics, Georgia College
and State University, USA.} \centerline{Campus Box 17,
Milledgeville} \centerline{ {E-mail address}: guy.biyogmam@gcsu.edu}
\bigskip
 \centerline{\small $^{3}$ Department of Mathematics, Faculty of sciences-University of Dschang}
  \centerline{Campus Box (237)67 Dschang, Cameroon}
\centerline{ {E-mail address}: bellnarcisse3@gmail.com}
\bigskip
 \centerline{\small $^{4}$ Department of Mathematics, Faculty of science-University of Dschang}
    \centerline{Campus Box (237)67 Dschang, Cameroon}
     \centerline{ {E-mail address}: batkamjacky3@yahoo.com }
\bigskip

\date{}

\bigskip \bigskip \bigskip

{\bf Abstract:}

In this work, we introduce a new class of Leibniz algebras, named
quasi-Artinian Leibniz algebras, which generalizes the minimal
condition on ideals. Furthermore, we provide some characterizations
 and give conditions under which a
quasi-Artinian Leibniz algebra is Artinian. Finally, within the
framework of Leibniz algebras, we establish a connection between
prime ideals and the quasi-Artinian structure.

\bigskip


\bigskip

\bigskip

\section{Introduction}
The concept of Leibniz algebra originated in works published in 1965 by Bloh \cite{b1}, and did
not receive significant attention until it was independently rediscovered by Jean Louis Loday \cite{LD}. Essentially, A Leibniz algebra is a
vector space equipped with a binary operation which has the property of being a derivation of
itself. Consequently, Leibniz algebras  extend the class of  Lie algebras,
and much of the research on Leibniz algebras focuses on developping
analogues of results from Lie theory on Leibniz algebras.

This paper studies the Artinian conditions on Leibniz algebras. The concept of Artinian algebras plays an important role in the study of
infinite-dimensional algebraic structures. It originated in the foundational work of
E. Artin and E.  Noether whose work on
chain conditions laid the groundwork for modern ring theory. In particular, Noether's introduction \cite{EN}
of the ascending chain condition initiated the study of finiteness
conditions in algebra and  Artin's examination   \cite{EA} of descending chain conditions
led to tools useful in controlling the ideal structure of rings \cite{A_{1}, key-3, B_{1}, key-5, C_{1}}.
 These conditions have become central in extending
finiteness methods  to more general algebraic systems ,
including Lie algebras, especially in the infinite-dimensional case \cite{key-1, D_{2}, D_{4}, key-10, key-13}.

 An Artinian Leibniz algebra is defined as a Leibniz algebra that satisfies the descending
chain condition on both left and right ideals. This provides a useful framework for understanding their structure, particularly in the infinite-dimensional case. As for many algebraic structures, this advantage can be extended to a broader class of Leibniz algebras by considering a suitable generalization of the Artinian condition. This motivates the introduction in this paper of the class of quasi-Artinian Leibniz
algebras, which generalizes both Artinian Leibniz algebras (\cite{OSBOR}, \cite{key-10}) and quasi-Artinian Lie algebras \cite{key-15}, and is shown to include important classes of Leibniz algebras, such as solvable algebras.  We investigate several results analogous to those known for quasi-Artinian Lie algebras.


The paper is organized as follows. In Section 2, we recall the necessary
definitions and set the notations used throughout the paper.
In Section 3,
we introduce the notion of quasi-Artinian Leibniz algebra in the category of Leibniz algebras and provide some examples and characterizations of these algebras . Among the results obtained, we show that the class of quasi-Artinian Leibniz algebras is closed under
quotients but not under extensions.
 In Section 4, we study the connection  between the subclass of Leibniz algebras satisfying the minimal condition and
  the subclass of Leibniz algebras satisfying the minimal condition on abelian ideals. Finally,
 in section 5,  we investigate infinite dimensional Leibniz algebras with chain condition(CC) on ideals of ideals.

\section{Preliminaries on Leibniz algebras}

We follow the notations and terminologies in $\cite{key-1}$, and
some of the notions are equally found in  $\cite{key-10}$ and
$\cite{key-11}$. For convenience, we recall some preliminaries.
\begin{De}
    A  \emph{ (left) Leibniz algebra}  is a vector space ${\Lieg}$  equipped with a bilinear map $[-,-] : \Lieg \otimes \Lieg \to \Lieg$,
    usually called \emph{Leibniz bracket},  satisfying the left \emph{Leibniz identity}:
    \[
    [x,[y,z]]= [[x,y],z]+[y,[x,z]],
    \]
    for all $x, y, z \in \Lieg.$
\end{De}

\indent Let $\Lieg$ be a Leibniz algebra of arbitrary dimension.
For a given ordinal $\alpha$, $\Lieg^{\left(\alpha\right)}$ denotes
the $\alpha$-th terms of the (transfinite) derived series defined
inductively by:
$\Lieg=\Lieg^{\left(0\right)},\Lieg^{\left(\alpha+1\right)}=\left[\Lieg^{\left(\alpha\right)},\Lieg^{\left(\alpha\right)}\right]$
and
$\Lieg^{\left(\rho\right)}=\underset{\alpha<\rho}{\cap}\Lieg^{\left(\alpha\right)}$
for every limit ordinal $\rho.$\\  The right and the left centers of
$\Lieg$ are respectively defined as follows:
$$Z^{r}_{\Lieg} = \{x \in \Lieg: [y, x] = 0, ~ \mbox{for each
element}~ y\in \Lieg\},$$
$$Z^{l}_{\Lieg} = \{x \in \Lieg: [x, y] = 0, ~ \mbox{for each
element}~ y\in \Lieg\}.$$ Thus the center of $\Lieg$ is given by:
$Z_{\Lieg}= Z^{r}_{\Lieg} \cap Z^{l}_{\Lieg}.$  \\
 Starting from the notion of the center of the Leibniz algebra $\Lieg$, for a given ordinal $\alpha,$ we denote by $\zeta_{\alpha}$,
the $\alpha$-th term of the transfinite upper central series of
$\Lieg:$ $\zeta_{0}(\Lieg)\subseteq \zeta_{1}(\Lieg) \subseteq
\zeta_{2}(\Lieg)\subseteq \cdots \subseteq \zeta_{\lambda}(\Lieg)=
\zeta_{\infty}(\Lieg )$ and define it inductively as follows:\\
$\zeta_{0}(\Lieg)= 0, \zeta_{1}(\Lieg) = Z_{\Lieg},$
$\zeta_{\alpha+1}\left(\Lieg \right)/\zeta_{\alpha}\left(\Lieg
\right)=\zeta_{1}\left(\Lieg /\zeta_{\alpha}\left(\Lieg
\right)\right)$ and $\zeta_{\rho}\left(\Lieg
\right)=\underset{\alpha<\rho}{\cup}\zeta_{\alpha}\left(\Lieg\right),$
for every limit ordinal $\rho$. The last term, $\zeta_{\infty}(\Lieg ),$
of this series is called the upper hypercenter of $\Lieg$ while the
hypercenter  of $\Lieg$ is $\zeta_{\ast}\left(\Lieg \right)=
\underset{\alpha}{\cup}\zeta_{\alpha}\left(\Lieg \right).$ A Leibniz
algebra $\Lieg$ is said to be hypercentral  if it coincides with its
hypercenter.\\
\indent Denote respectively by $\leq$ and $\vartriangleleft$ the
relation  of subalgebra and that of ideal of $\Lieg,$ and let
$\mathfrak{X}$ be the collection of Leibniz algebras up to
isomorphism, together with the $0$-dimensional Leibniz algebra.
Define the following classes:
\begin{equation*}
\begin{aligned}
{}{S}\mathfrak{X}&: = \{ H: \exists \Lieg \in \mathfrak{X} ~ \mbox{and} ~ H \leq \Lieg \}. \\
{}{E}\mathfrak{X}& := \{ \Lieg: \exists \quad
0=\Lieg_{0}\vartriangleleft\cdots\vartriangleleft
 \Lieg_{n}=\Lieg, \quad \mbox{a finite series,}  \quad
 \quad \\ & ~ \quad
\mbox{such that} \quad \Lieg_{i+1}/\Lieg_{i}\in\mathfrak{X}, \quad
\mbox{for} \quad 0\leq i\leq n-1\}. \\
{}{R}\mathfrak{X}& :=~ \{ \Lieg: \exists \quad
\left(I_{\alpha}\right)_{\alpha\in\varLambda}, \quad
I_{\alpha}\vartriangleleft \Lieg, \quad
\Lieg/I_{\alpha}\in\mathfrak{X}\\ & ~~ \quad \mbox{for all}\quad
\alpha \in \varLambda \quad \mbox{and}\quad
\underset{\alpha\in\varLambda}{\cap}I_{\alpha}= Leib(\Lieg)\}.\\
{}\mathfrak{X}^{Q}& :=~ \{ \Lieg: \quad \forall I\vartriangleleft
\Lieg, \quad
\Lieg/I\in\mathfrak{X}\}.\\
\end{aligned}
\end{equation*}
We say that a class $\mathfrak{X}$ of Leibniz algebras is $E$-$closed$
(respectively $R$-$closed$) if ${E}\mathfrak{X}=\mathfrak{X}$
(respectively ${R}\mathfrak{X}=\mathfrak{X}$).

  Throughout the remainder of this paper, we adopt the following notation for these classes of Leibniz
  algebras:
\begin{equation*}
\begin{aligned}
\mathfrak{U} &  := ~ abelian.\\
\mathfrak{U}^{d} & :=~ solvable ~ of ~ derived ~ length \leq d.\\
\mathfrak{N} &  := ~nilpotent.\\
\mathfrak{N}_{c} &:=~ nilpotent ~ of ~ class \leq c.\\
E\mathfrak{U} & := ~ solvable.\\
\;{\L}\mathfrak{N}&:= ~ locally ~ nilpotent.\\
  \mathfrak{3}& := ~Leibniz ~ algebra ~\Lieg ~such~ that ~\zeta_{\ast}\left(\Lieg
  \right)=\Lieg.\\
\end{aligned}
\end{equation*}
\begin{Rem}\label{22}
\begin{enumerate}
 \item  If $\mathfrak{R}$ denotes the subclass of Leibniz algebras $\Lieg$ whose
elements satisfy: $ x\in \Lieg$ implies $x=0$ or
$x\notin\left[x,\Lieg \right] \cup \left[\Lieg,x \right]$. Then
${\L}\mathfrak{N}\subseteq\mathfrak{R}$.

\item Moreover if one denotes by $\;\min\mbox{-}\vartriangleleft$ the class
of Leibniz algebras together with the minimal condition, then we
have the following subclasses:
\begin{equation*}
\begin{aligned}
\min\mbox{-}\vartriangleleft{}_{E\mathfrak{U}}&:= ~ the ~ subclass ~ of
~ Leibniz ~ algebras ~satisfying ~ the ~ minimal ~ condition \\
&~
~on ~ solvable ~ ideals.\\
\min\mbox{-}\vartriangleleft\mathfrak{U}&:= ~ the ~ subclass ~ of ~ Leibniz
~ algebras ~satisfying ~ the ~ minimal ~ condition
\\&~~ on ~ abelian ~
ideals.\\
\end{aligned}
\end{equation*}
These two subclasses are related by:
$\min\mbox{-}\vartriangleleft{}{E}\mathfrak{U} \subsetneq
\min\mbox{-}\vartriangleleft\mathfrak{U}$ since every abelian two-sided ideal is solvable, but the converse is not true.
\end{enumerate}
\end{Rem}
\section{Quasi-Artinian structure on Leibniz algebras.}
\subsection{Definitions and some characterizations}
Hereafter, we extend the well-known notion of the minimal condition on ideals from Lie algebras to the category of Leibniz algebras, and we provide several characterizations of Leibniz algebras satisfying this condition.


\begin{De}
    A Leibniz algebra $\Lieg$ is said to be a left (resp. right) \textit{quasi-Artinian} if, for every descending chain of ideals
    \begin{equation*}
    I_{1}\supseteq I_{2}\supseteq \cdots \supseteq I_{t}\supseteq \cdots,
       \end{equation*}
    there exists $ m_{l} ~ (\mbox{resp.} ~ m_{r}) \in \mathbb{N} $ such
    that for all $n\in \mathbb{N},$
    \begin{equation*}
        \left[\Lieg^{(m_{l})}, I_{m_{l}}\right] \subseteq I_{n } ~(resp. \left[I_{m_{r}}, \Lieg^{(m_{r})}\right] \subseteq I_{n}).
    \end{equation*}
    We say that a Leibniz algebra $\Lieg$ is \textit{quasi-Artinian} if it is both left and right \textit{quasi-Artinian}.
\end{De}

\noindent\textbf{Notations:}
The class of left (resp.\ right) quasi-Artinian Leibniz algebras is denoted by
$q\min\text{-}\lhd_{l}$ (resp.\ $q\min\text{-}\lhd_{r}$), and the class of all quasi-Artinian Leibniz algebras is denoted by $q\min\text{-}\lhd$.

Hereafter, we present some observations and characterizations of quasi-Artinian Leibniz algebras.

\begin{Rem}
\begin{enumerate}
\item A Leibniz algebra $\Lieg$ is said to be Artinian if any decreasing
sequence of ideals of $\Lieg$ terminates. Consequently any Artinian
Leibniz algebra is quasi-Artinian.
\item Any finite dimensional Leibniz algebra is Artinian, and hence quasi-Artinian.
\item Let $I$ be a two-sided ideal of a Leibniz algebra $\g$ such
that $\frac{\Lieg}{I}$ is Artinian, then $\Lieg$ is quasi-Artinian.
\item A Leibniz algebra $\Lieg$ is  quasi-Artinian if and
only if for every descending chain
    of ideals, $$I_{1}\supseteq I_{2}\supseteq \cdots \supseteq I_{t}\supseteq
    \cdots,$$
    of $\Lieg,$ there exists positive integer $s$ such that:
    $$\left\{\begin{array}{lll}
    \left[\Lieg^{\left(s\right)},I_{s}\right]\subseteq\underset{i\geq1}{\cap}I_{i}\\
    \\
    \left[I_{s},\Lieg^{\left(s\right)}\right]\subseteq\underset{i\geq1}{\cap}I_{i}
    \end{array}
    \right.$$
     \end{enumerate}

\end{Rem}

\begin{Pro}

    A Leibniz algebra $\Lieg$ is quasi-Artinian if and only if there exists $I$, a two-sided ideal of $\Lieg$ such that $\Lieg/I $ is solvable
    and the sets of ideals $ \{[\mathcal{H},I] :
\mathcal{H} \lhd \Lieg\} $ and $ \{[I,\mathcal{H}] : \mathcal{H}
\lhd \Lieg\} $ satisfy the minimal condition.
\end{Pro}

\begin{proof}
 Assume that $\Lieg \in q\min\mbox{-}\lhd$ and set $I= \Lieg^{(\omega)}$ with $w$ the smallest limit ordinal. Then $I$ is a two-sided ideal of $\Lieg.$
 Since $(\Lieg^{(k)})_{k\geq 0}$ is a decreasing sequence of ideals of $\Lieg,$ there exists $n< w$ such that
$\Lieg^{(n+1)}\subseteq \Lieg^{(n)}=I$.
So $(\frac{\Lieg}{I})^{(n+1)}=0,$ and hence $\Lieg/I$ is solvable.
 Now consider the sequence of ideals $(\mathcal{H}_{i} \lhd \Lieg)_{i \geq 1}$ such that
     $[\mathcal{H}_{1},I] \supseteq [\mathcal{H}_{2},I] \supseteq \cdots \supseteq [\mathcal{H}_{n},I] \supseteq \cdots.$
Since $\Lieg \in q \min\mbox{-}\lhd$, there exists an integer $r \geq 1$ such that $[\Lieg^{(r)},[\mathcal{H}_{r},I]] \subseteq [\mathcal{H}_{i},I]$ and
  $[[\mathcal{H}_{r},I],\Lieg^{(r)}] \subseteq [\mathcal{H}_{i},I]$, for all $i\geq
  1$. Notice that $\omega$ is a limit ordinal, thus $[\Lieg^{(\omega)}, \Lieg^{(\omega)}]=I^{(1)}=
  \Lieg^{(\omega)}=I.$ Hence for $n<\omega,$ $I= \Lieg^{(n)}=
  \Lieg^{(\omega)}.$ So if $m= \mbox{max}(\{n, r\}),$ then
   $[\mathcal{H}_m, I]= [\mathcal{H}_m, [I, I]]\subseteq [I,[\mathcal{H}_m, I]]= [\Lieg^{(m)},[\mathcal{H}_m, I]]\subseteq [\mathcal{H}_i,
   I],$ for all $i \in \mathbb{N}.$ So the set $ \{[\mathcal{H},I] :
\mathcal{H} \lhd \Lieg\} $ satisfies the minimal condition. Similarly, one shows that $ \{[I,\mathcal{H}] : \mathcal{H} \lhd
\Lieg\} $ also  satisfies the minimal condition.\\
Conversely, let $I$ be a two-sided ideal of $\Lieg$ such that
$\Lieg/I $ is solvable and consider the descending chain of ideals
of $\Lieg$:
$\mathcal{J}_{1}\supseteq \mathcal{J}_{2} \supseteq \cdots \supseteq
\mathcal{J}_{n}\supseteq \cdots.$
Then there exists $n\in \mathbb{N}$ such that $\Lieg^{(n)}= I$ and
the set of ideals: $ \{[I,\mathcal{J}_{i}] : i\geq 1\}$ has a
minimal element $[I,\mathcal{J}_{s}].$ Set $m= \mbox{max}\{s, n\},$
we have $[\Lieg^{(m)},\mathcal{J}_{m}]\subseteq
[I,\mathcal{J}_{s}]\subseteq [I,\mathcal{J}_{i}]\subseteq
\mathcal{J}_{i}$, for all $ i\geq 1.$ That is $\Lieg$ is a left
quasi-Artinian Leibniz algebra. Similarly, one shows that
$\Lieg$ is a right quasi-Artinian Leibniz algebra. Therefore
$\Lieg$ is a quasi-Artinian Leibniz algebra.
\end{proof}

\subsection{On Quasi-Artinian Leibniz algebras} The main goal of this subsection
is to establish the connection between some fundamental
properties of Leibniz algebras and the minimal condition of ideals
on Leibniz algebras.

\begin{Pro}\label{Prop 3.5}
Let $\Lieg$ be a Leibniz algebra, then the following statements
hold.
\begin{enumerate}
\item If $\Lieg$ is solvable, then it is quasi-Artinian.
\item If $\Lieg$ is hypercentral and quasi-Artinian, then $\Lieg$ is solvable.
\end{enumerate}
\end{Pro}
\begin{proof}
Let $\Lieg$ be a Leibniz algebra.\\
 To prove the first statement, suppose that $\Lieg$ is solvable. Then there exists $m_0\in
\mathbb{N}$ such that $\Lieg^{(m_0)}=0$. For any descending chain of
ideals of $\Lieg,$ $I_{0} \supseteq I_{1} \supseteq I_{2} \supseteq
\cdots \supseteq I_{t} \supseteq \cdots$, one has $[\Lieg^{(m_0)},
I_{m_0}]= [I_{m_0}, \Lieg^{(m_0)}]=0 \in
    I_n,$ for all $n\in \mathbb{N}.$ So $\Lieg$ is a
    quasi-Artinian Leibniz algebra.\\
 To prove the second statement,  let $ \Lieg^{(1)} \supseteq \Lieg^{(2)} \supseteq
\cdots \supseteq \Lieg^{(t)} \supseteq \cdots$ be a descending chain
of ideals of a hypercentral Leibniz algebra $\Lieg$. Observe that
the transfinite  hypercentrality of $\Lieg$ implies
$\bigcap_{\alpha< \omega}\Lieg^{(\alpha)} = \Lieg^{(\omega)}=0$ with
$\omega$, the smallest limit ordinal.
Moreover if $\Lieg$ is quasi-Artinian, there
exists $m \in \mathbb{N}$ such that
$g^{(m+1)}=[\Lieg^{(m)},\Lieg^{(m)}] \subseteq \Lieg^{(n)}$ for all
$n \geq 0$. Therefore
    $\Lieg^{(1)} \supseteq \Lieg^{(2)} \supseteq \cdots \supseteq\Lieg^{(m)} \supseteq \Lieg^{(m+1)}=\Lieg^{(m+2)}=\cdots,$
and thus $g^{(m+k)}=\Lieg^{(w)}=0$ for ll $k\geq 1$. Hence  $\Lieg$ is
solvable.
\end{proof}

\begin{Co}
A Leibniz algebra $\Lieg$ is  quasi-Artinian if one of the following
assertion holds:
\begin{enumerate}
\item  $\Lieg$ has a solvable  ideal $H$ such that $\frac{\Lieg}{H}$ is solvable.

\item $\Lieg$ has finitely many subalgebras $\{H_{i}: 0\leq i\leq
n\} $ such that $0= H_{0}\vartriangleleft H_{1} \vartriangleleft
\cdots \vartriangleleft H_{n}= \Lieg$  and such that
$\frac{H_{i+1}}{H_{i}}$ is  abelian for $0\leq i\leq n-1.$
\end{enumerate}
\end{Co}

\begin{proof}
On one hand, let $H$ be a solvable  two-sided ideal of a Leibniz algebra $\Lieg$ such that $\frac{\Lieg}{H}$ is
solvable. There exists $0 \neq n\in \mathbb{N}$ such that
$\Lieg^{(n)}= H$ and since $H$ is solvable, $\Lieg$ is also solvable
and the result follows by Proposition \ref{Prop 3.5}.\\
On the other hand, assume that $\Lieg$ has  finitely many subalgebras $\{H_{i}: 0\leq i\leq
n\} $ such that $0= H_{0}\vartriangleleft H_{1} \vartriangleleft
\cdots \vartriangleleft H_{n}= \Lieg$  and such that
$\frac{H_{i+1}}{H_{i}}$ is  abelian for $0\leq i\leq n-1.$ Then $H_{1}$
is obviously solvable since $\frac{H_{1}}{H_{0}}$ is abelian.  Now assume
that $H_{k}$ is solvable  for $2\leq i\leq n-1.$ Since
$\frac{H_{n}}{H_{n-1}}$ is abelian, then $\Lieg^{(1)}= H_{n-1}$ and
by the induction hypothesis, $\Lieg$ is solvable. The result follows again by Proposition \ref{Prop 3.5}.
\end{proof}

\begin{Th}
Let $\Lieg$ be a Leibniz algebra over an arbitrary field of characteristic different to two.
   The following assertions are equivalent:

\begin{enumerate}
    \item[(i)] $\Lieg$ is quasi-Artinian,
    \item[(ii)] There exists $m\in\mathbb{N}$ such that for every descending chain
    of ideals $I_{1}\supseteq I_{2}\supseteq\cdots$
 of \Lieg, the descending chains     of ideals $\left[\Lieg^{\left(m\right)},
 I_{1}\right]\supseteq\left[\Lieg^{\left(m\right)},I_{2}\right]\supseteq\cdots$
 and $\left[I_{1},\Lieg^{\left(m\right)}\right]\supseteq\left[I_{2},\Lieg^{\left(m\right)}\right]\supseteq\cdots$
    terminate.
    \item[(iii)] For any non-empty collection $\mathcal{C}$ of ideals of $\Lieg$, there
    exists $I \in \mathcal{C}$ and $m \in \mathbb{N}$ such that $\left[\Lieg^{\left(m\right)},I\right]
    \subseteq J$ and $\left[I,\Lieg^{\left(m\right)}\right]\subseteq J$
    for all $J \in \mathcal{C}$ with $J\subseteq I$.
\end{enumerate}\end{Th}

\begin{proof} To prove that
$(i)\Rrightarrow (ii),$
    assume that $\Lieg \in q \min$-$\lhd$. Since $\Lieg \supseteq \Lieg^{(1)} \supseteq \Lieg^{(2)} \supseteq \cdots \supseteq \Lieg^{(t)} \supseteq \cdots$
    is a descending chain of ideals of $\Lieg$, there exists $ s \in \mathbb{N}$ such that $[\Lieg^{(s)},\Lieg^{(s)}] \subseteq \Lieg^{(r)}$
    for all $r \in \mathbb{N}$. So, $\Lieg^{(s+1)}=[\Lieg^{(s)},\Lieg^{(s)}] \subseteq \Lieg^{(r)}$ for all $r \in \mathbb{N}$.
     Hence $\Lieg^{(s+1)}=\Lieg^{(s+k)}$ for all $k\geq 1.$
    In addition, $I_{1} \supseteq I_{2} \supseteq I_{3} \supseteq \cdots$ is a descending chain of ideals of $\Lieg$, so there exists $t \in \mathbb{N}$
    such that $[\Lieg^{(t)},I_{t}] \subseteq I_{k}$ for all $k \in \mathbb{N}$. Set $p= \mbox{max}\{s+1, t\}\ $;
     one has   $[\Lieg^{(p)},I_{t}] = [\Lieg^{(s+1)},I_{t}]\subseteq
     I_{k}$,  for all $k \in \mathbb{N},$  thus $[\Lieg^{(s+1)},I_{t}]\subseteq [\Lieg^{(s+1)},I_{k}],$ for all $k \in
     \mathbb{N}.$ Taking $m_{1}= s+1,$ we obtain $[\Lieg^{(m_{1})},I_{t}]\subseteq [\Lieg^{(m_{1})},I_{k}],$ for all $k \in
     \mathbb{N}.$ That is, $[\Lieg^{(m_{1})},I_{t}]= [\Lieg^{(m_{1})},I_{t+i}]$ for all $i \in
     \mathbb{N}.$ Henceforth the descending chain of ideals
     $\left[\Lieg^{\left(m_{1}\right)},I_{1}\right]\supseteq\left[\Lieg^{\left(m_{1}\right)},I_{2}\right]\supseteq\cdots$
     terminates. Similarly, we obtain $m_{2}$ such that $[I_{1},\Lieg^{(m_{2})}] \supseteq [I_{2},\Lieg^{(m_{2})}] \supseteq \cdots$ terminates.
Finally for $m=max\{m_{1}, m_{2}\},$ the above two sequences of
ideals terminate simultaneously.\\
To prove that $(ii)\Rrightarrow (iii),$
 let $\mathcal{C}$ the collection of all non trivial ideals of $\Lieg.$ Also,  let $0\neq m\in \mathbb{N}$ and  $I_{0}\in \mathcal{C},$ and
assume that $(iii)$ does not hold. Then, there exists $I_{1}\in \mathcal{C}$ such that $I_{0}\supseteq I_{1}$
 with $[I_{0},\Lieg^{(m)}] \supsetneq [I_{1},\Lieg^{(m)}]$ or $[\Lieg^{(m)}, I_{0}] \supsetneq [\Lieg^{(m)}, I_{1}].$
 Similarly for $I_{1},$ there exists $I_{2}\in \mathcal{C}$ such that
 $I_{1}\supseteq I_{2}$ and  $[I_{1},\Lieg^{(m)}] \supsetneq
 [I_{2},\Lieg^{(m)}]$ or $[\Lieg^{(m)}, I_{1}] \supsetneq [\Lieg^{(m)}, I_{2}].$ So,  we obtain progressively a non
 stationary descending chain of ideals in $\mathcal{C},$ $I_{0}\supseteq I_{1}\supseteq
 I_{3}\cdots,$ from which one can extract an infinite descending
 subchain, $I^{\prime}_{0}\supseteq I^{\prime}_{1}\supseteq
 I^{\prime}_{2}\cdots,$ with terms in $\mathcal{C}$ such that one of the following descending chain
$[I^{\prime}_{0},\Lieg^{(m)}] \supsetneq
[I^{\prime}_{1},\Lieg^{(m)}] \supsetneq
 [I^{\prime}_{2},\Lieg^{(m)}]\supsetneq \cdots,$ $[\Lieg^{(m)}, I^{\prime}_{0}] \supsetneq [\Lieg^{(m)}, I^{\prime}_{1}] \supsetneq
 [\Lieg^{(m)}, ]\supsetneq \cdots$ does not terminate. Therefore $(ii)$ does not
 hold as well.\\
To prove that $(iii)\Rrightarrow (i),$
let  $I_{1} \supseteq I_{2} \supseteq I_{3} \supseteq \cdots$ be a
descending chain of ideals of $\Lieg.$ Set $\mathcal{C}= \{I_{s},
s>0\}$ the set whose elements  are the terms of the above sequence.
From $(iii),$ there exists $I_{s_{0}}\in\mathcal{C}$ and $r\in
\mathbb{N}$ such that $[\Lieg^{(r)}, I_{s_{0}}]\subseteq I_{s}$ and
$[I_{s_{0}}, \Lieg^{(r)}]\subseteq I_{s},$ for all $s\geq s_{0}.$
Therefore $[\Lieg^{(r)}, I_{s_{0}}]\subseteq I_{s}$ and $[I_{s_{0}},
\Lieg^{(r)}]\subseteq I_{s},$ for all $s>0.$ Considering $m=
\mbox{max}\{r, s_{0}\}$, one has $[\Lieg^{(m)}, I_{m}]\subseteq
I_{s}$ and $[I_{m}, \Lieg^{(m)}]\subseteq I_{s},$ for all $s>0.$ Hence
 $\Lieg$ is a quasi-Artinian Leibniz algebra.
\end{proof}

\begin{Rem} Using the above theorem, one shows that the subclass $q\min\mbox{-}\lhd$ of quasi-Artinian
Leibniz algebra is not $E$-closed. Indeed, let $\LieH$ be the infinitely
generated $\mathbb{Q}$-vector space defined by:
$$\LieH:=\langle\{x_{\alpha}: 0\neq\alpha\in \mathbb{Q}\}\rangle.$$
Endows $\LieH$ with the abelian Leibniz algebra structure  and
define the following vector space homomorphisms:
$$\LieH\stackrel{a}\longrightarrow \LieH, x_{\alpha}\longmapsto
x_{\alpha+1}; \LieH\stackrel{b}\longrightarrow \LieH,
x_{\alpha}\longmapsto (\alpha-1)x_{\alpha-1};
\LieH\stackrel{c}\longrightarrow \LieH, x_{\alpha}\longmapsto
2\alpha x_{\alpha}.$$ Set $\Lieg= \LieH\oplus \LieK$ where $\LieK$
is the vector space generated by $\{a, b, c\}.$ Endows
$\Lieg$ with the Leibniz bracket by extending the bracket of $\LieH$
as follows: $[x_{\alpha}, v]= v(x_{\alpha}),$ $[v, x_{\alpha}]= 0$
and $[u,v]= u\circ v -v\circ u,$ for $u, v\in\{a, b, c\}.$\\ We
have the following facts: $\LieH\lhd\Lieg,$ $(\LieH, \LieK\cong
\frac{\Lieg}{\LieH} \in
q\min\mbox{-}\lhd)$  and thus $\Lieg \in E(q\min\mbox{-}\lhd).$ \\
Consider the descending sequence $(\LieH_{n})_{n>0}$ of subspaces of
$\LieH$ such that $\LieH_{n}= \langle\{x_{\alpha}: \alpha\in
\mathbb{Q}, \alpha<\frac{1}{n}\}\rangle.$ Obviously
$\LieH_{n}\lhd\Lieg,$ for $n>0$   and for some positive integer $m,$
$\Lieg^{(m)}= \Lieg.$ In addition, the descending sequence,
$\{[\LieH_{n},\Lieg^{(m)}]= \LieH_{n} : n>0\},$ of ideals of $\Lieg$
does not satisfy the minimal condition. It follows from the above
theorem that $\Lieg\notin q\min\mbox{-}\lhd$ and hence $q\min\mbox{-}\lhd$ is not
$E$-closed.
\end{Rem}
\begin{Pro}
The subclass of quasi-Artinian Leibniz algebras, $q\min\mbox{-}\lhd,$ is $Q$-closed.
\end{Pro}
\begin{proof}
 Let $J$ be an ideal of a quasi-Artinian Leibniz algebra $\Lieg$ and consider $ \Lieg \stackrel{\phi}\rightarrow
\frac{\Lieg}{J}$ the natural projection. Let $\bar{J_{1}} \supseteq
\bar{J_{2}} \supseteq \cdots \supseteq \bar{J_{r}} \supseteq \cdots$
be a descending chain of ideals of $\bar{\Lieg}=\frac{\Lieg}{J}.$ Then
$\phi^{-1}(\bar{J_{1}}) \supseteq \phi^{-1}(\bar{J_{2}}) \supseteq
\cdots \supseteq \phi^{-1}(\bar{J_{r}}) \supseteq \cdots$ is a
descending chain of ideals of $\Lieg$.
Since $\Lieg$ is quasi-Artinian, there exists $m \in \mathbb{N}$
such that $[\Lieg^{(m)},\phi^{-1}(\bar{J_{m}})] \subseteq
\phi^{-1}(\bar{J_{n}})$ and $[\phi^{-1}(\bar{J_{m}}),\Lieg^{(m)}]
\subseteq \phi^{-1}(\bar{J_{n}})$ for all  $n \in \mathbb{N}$.
Applying the map $\phi$ to the previous inclusions leads us to the
following: $[(\phi(\Lieg))^{(m)},\bar{J_{m}}] \subseteq \bar{J_{n}}$
and $[\bar{J_{m}},(\phi(\Lieg))^{(m)}] \subseteq \bar{J_{n}}$  for
all $n \in \mathbb{N}$ and thus $[\bar{\Lieg}^{(m)},\bar{J_{m}}]
\subseteq \bar{J_{n}}$
 and $[\bar{J_{m}},\bar{\Lieg}^{(m)}] \subseteq \bar{J_{n}}$  for all  $n \in \mathbb{N}.$ Hence  $q\min\mbox{-}\lhd$ is $Q$-closed.
\end{proof}
\begin{Th}
  Let $I$ be an ideal of  a Leibniz algebra $\Lieg$. $\Lieg$ is
quasi-Artinian if one of the following assertions holds:
\begin{enumerate}
    \item $\frac{\Lieg}{I}$ is quasi-Artinian and $I$ is Artinian,
    \item $I$ is quasi-Artinian and $\frac{\Lieg}{I}$ is solvable,

\end{enumerate}

\end{Th}
\begin{proof}
To prove the first part, let $I$ be an Artinian  ideal of $\Lieg$ such that $\frac{\Lieg}{I}$ is quasi-Artinian and consider
    $J_{0} \supseteq J_{1} \supseteq J_{2} \supseteq\cdots$, a descending chain of ideals of $\Lieg.$ Then $J_{0} \cap I
\supseteq J_{1} \cap I \supseteq J_{2} \cap I \supseteq \cdots
\supseteq J_{r} \cap I \supseteq \cdots$ is a descending chain of
ideals of $I$ and $\frac{J_{0}+I}{I} \supseteq \frac{J_{1} +I}{I}
\supseteq \frac{J_{2}+I}{I} \supseteq \cdots \supseteq
\frac{J_{r}+I}{I} \supseteq \cdots$ is a descending chain of ideals
of $\frac{\Lieg}{I}$. On one hand, since $I$ is Artinian, the
descending chain of ideals $\{J_{s}\cap I, s\in \mathbb{N}\}$
terminates. Set $J_{s_{0}} \cap I$ be its minimal element. On the
other hand, since $\Lieg/I$ is quasi-Artinian, there exists $m \in
\mathbb{N}$ such that $([\Lieg^{(m)},J_{m}] + I)/I \subseteq (J_{n}
+ I)/I$ and $([J_{m},\Lieg^{(m)}] + I)/I \subseteq (J_{n} + I)/I$
for all $n \in \mathbb{N}$. So $[\Lieg^{(m)},J_{m}] \subseteq J_{n}
+ I$ and $[J_{m},\Lieg^{(m)}] \subseteq J_{n} + I$ for all $n \in
\mathbb{N}.$ By taking $q= \mbox{max}\{m, s_{0}\},$ we have
$[\Lieg^{(q)},J_{q}] \subseteq J_{q} \subseteq J_{s_{0}-i}$ and
$[J_{q}, \Lieg^{(q)}] \subseteq J_{q} \subseteq J_{s_{0}-i},$ for
$0\leq i\leq s_{0}.$ Moreover since $J_{s_{0}}\cap I =
J_{s_{0}+i}\cap I$ and $J_{s_{0}+i}\cap J_{q} + J_{s_{0}+i}\cap I=
J_{\mbox{max}\{s_{0}+i , q\}} + J_{s_{0}}\cap I=
J_{\mbox{max}\{s_{0}+i , q\}}, \quad \mbox{for all}\quad i> 0,$ we
obtain $[\Lieg^{(q)},J_{q}] \subseteq J_{s_{0}+i}$ and $[J_{q},
\Lieg^{(q)}] \subseteq J_{s_{0}+i},$ for all $i> 0.$ Therefore
$[\Lieg^{(q)},J_{q}] \subseteq J_{n}$ and $[J_{q}, \Lieg^{(q)}]
\subseteq J_{n},$ for all $i> 0.$ Hence $\Lieg$ is quasi-Artinian.

To prove the second part,
assume that $I$ is quasi-Artinian and $\frac{\Lieg}{I}$ is solvable, and
consider $J_{0} \supseteq J_{1} \supseteq J_{2} \supseteq \cdots,$ a
descending chain of ideals of $\Lieg.$  So $J_{0}\cap I \supseteq
J_{1}\cap I \supseteq J_{2}\cap I \supseteq \cdots$ a descending
chain of ideals of $I$. Since $I$ is quasi-Artinian, there exists
$r\in \mathbb{N}$ such that $[I^{(r)}, J_{r}\cap I]\subseteq
J_{s}\cap I$ and $[J_{r}\cap I, I^{(r)}]\subseteq J_{s}\cap I,$ for
all $s\in \mathbb{N}$. In addition, since $\frac{\Lieg}{I}$ is solvable,
there exists  $p>0$ such that $\Lieg^{(p)}= {I}.$ Observe that
$[[\Lieg^{(p+r)}, \Lieg^{(p+r)}], J_{p+r}]= [\Lieg^{(p+r+1)},
J_{p+r}]$ and $[\Lieg^{(p+r)}, J_{p+r}]\subseteq I\cap J_{p+r};$ so
one has $[\Lieg^{(p+r+1)}, J_{p+r+1}]\subseteq [\Lieg^{(p+r+1)},
J_{p+r}]\subseteq [\Lieg^{(p+r)}, [\Lieg^{(p+r)},
J_{p+r}]]\subseteq[I^{(r)}, I\cap J_{p+r}]\subseteq [I^{(r)}, I\cap
J_{r}]\subseteq I\cap J_{s} \subseteq J_{s},$ for all $\geq 0.$
Similarly, one can show that $[ J_{p+r+1}, \Lieg^{(p+r+1)}]\subseteq
J_{s},$ for all $s\geq 0.$ Hence $\Lieg$ is quasi-Artinian.
\end{proof}

From the result above,  we deduce the following observations:
\begin{Co}\label{co3.7}
    Let $\Lieg$ be a Leibniz algebra and $I$ an ideal  of $\Lieg.$
    \begin{enumerate}
        \item If $I$ is simple and $\Lieg/I$ is solvable, then $\Lieg$ is quasi-Artinian.
        \item If $I$ is simple and $\Lieg/I$ is abelian, then $\Lieg$ is quasi-Artinian.
    \end{enumerate}

\end{Co}

Below are two examples of  Leibniz algebras; the first one is an
Artinian Leibniz algebra which is also quasi-Artinian while the
second one is  a quasi-Artinian Leibniz algebra that is not Artinian.

\begin{Ex}

   Consider $\Lieg$ the $\mathbb{K}$-vector space spanned by
    $\{e_{1}, e_{2}, e_{3}, e_{4}, e_{5}, e_{6}\}.$ Define the binary operation
    $\Lieg\otimes \Lieg\stackrel{[-, -]}\longrightarrow \Lieg$ as
    follows: $[e_{2}, e_{2}]= e_{1}, [e_{3}, e_{3}]= e_{4},  [e_{4}, e_{3}]= e_{5}, [e_{5}, e_{3}]= e_{6}.$
 One can verify that the above bracket satisfies the
    Leibniz identity. Set $I$ and $J$ the subspaces of
    $\Lieg$ respectively spanned by $\{e_{1}, e_{2}\}$ and $\{
    e_{3}, e_{4}, e_{5}, e_{6}\}$. Then $I$ and $J$ are two-sided ideals of
    $\Lieg.$ Moreover, $I$ is a simple ideal of $\Lieg$ and
    $\frac{\Lieg}{I} \cong J$ is a nilpotent Leibniz algebra. So
    $\frac{\Lieg}{I}$ is a solvable Leibniz algebra. It follows by Corollary \ref{co3.7}
    that $\Lieg$ is a quasi-Artinian Leibniz algebra, which is also  Artinian since it is of finite dimension.

\end{Ex}

\begin{Ex} Consider the infinite dimensional vector space $\Lieg$ with  basis elements $\{e_{1},e_{2}, \cdots\}$ and  define a binary operation on
    $\Lieg$ as follows: $[e_{1},e_{2}]=e_1$ and
    $[e_{i},e_{3}]= e_{i+1},$ $i\geq 4$. It is easy to verify that the above bracket satisfies the
    Leibniz identity, and hence $\Lieg$ is a Leibniz algebra. Note that $\Lieg$ is not an Artinian Leibniz algebra. Now, consider $I$ and $J$ the subspaces of $\Lieg$ spanned
    by $\{e_{1},e_{2}\}$ and
    $\{e_{3},e_{4},\cdots\}$ respectively. Clearly, $I$ and
    $J$ are two-sided ideals of $\Lieg$. In addition, $I$ is a simple
    ideal while $J$ is a solvable ideal of $\Lieg$. Since
    $\frac{\Lieg}{I}\cong J,$ it follows again by Corollary \ref{co3.7} that $\Lieg$
    is a  quasi-Artinian Leibniz algebra.

\end{Ex}

\begin{Pro}
A finite direct sum of quasi-Artinian Leibniz algebras is
    a quasi-Artinian Leibniz algebra.
\end{Pro}
\begin{proof}
    It suffices to prove that for any two quasi-Artinian Leibniz algebras $\Lieg_{1}$ and $\Lieg_{2}$,  $\Lieg=\Lieg_{1}\oplus\Lieg_{2}$
  is a  quasi-Artinian Leibniz algebra.
    Let $I_{1} \supseteq I_{2} \supseteq \cdots $
be a descending chain of ideals of $\Lieg.$ Let $I_{1}^{1}
\supseteq I_{2}^{1} \supseteq \cdots $ be a descending chain of
ideals of $\Lieg_{1}$ and $I_{1}^{2} \supseteq I_{2}^{2} \supseteq
\cdots$ be a descending chain of ideals of $\Lieg_{2}$ such that
$I_{n}= I^{1}_{n} \oplus  I^{2}_{n},$ for all $n\geq 1.$
Since $\Lieg_{1}$ is quasi-Artinian, then there exists $m_{1} \in
\mathbb{N}$ such that $[\Lieg_{1}^{(m_{1})},I_{m_{1}}^{1}] \subseteq
I_{n}^{1}$ and $[I_{m_{1}}^{1},\Lieg_{1}^{(m_{1})}] \subseteq
I_{n}^{1}$, for all $n\geq 1.$ Similarly, since $\Lieg_{2}$ is
quasi-Artinian, then there exists $m_{2} \in \mathbb{N}$ such that
$[\Lieg_{2}^{(m_{2})},I_{m_{2}}^{2}] \subseteq I_{n}^{2}$ and
$[I_{m_{2}}^{2},\Lieg_{2}^{(m_{2})}] \subseteq I_{n}^{2},$ for all
$n\geq 1.$ By setting $m=\max\{m_{1},m_{2}\}$, we have:
\begin{equation*}
    \begin{aligned}
        \left[\Lieg^{(m)}, I_{m}\right]  &
        = \left[(\Lieg_{1} \oplus \Lieg_{2})^{(m)}, I_{m}\right]\\
        & = \left[\Lieg_{1}^{(m)} \oplus \Lieg_{2}^{(m)},I_{m}\right]\\
            & = \left[\Lieg_{1}^{(m)},I_{m}^{1}\right] \oplus \left[\Lieg_{2}^{(m)},I_{m}^{2}\right]\\
            & \subseteq [\Lieg_{1}^{(m_{1})}, I_{m}^{1}] \oplus \left[\Lieg_{2}^{(m_{2})},
            I_{m}^{2}\right]\\
            & \subseteq I^{1}_{n} \oplus  I^{2}_{n}= I_{n}, for\quad
            all\quad n\geq 1.
    \end{aligned}
\end{equation*}
 Similarly, we have $[I_{m},\Lieg^{(m)}] \subseteq I_{n}$ for
all $n \in \mathbb{N}.$ Hence $\Lieg$ is a quasi-Artinian Leibniz algebra.
\end{proof}

\section{Connection between $q \min\mbox{-} \lhd$ and $\min\mbox{-} \lhd \mathfrak{U}$ }

In this section, we  present classes $\mathfrak{X}$ of Leibniz algebras
 such that $q\,\min\mbox{-}\vartriangleleft\cap\mathfrak{X}=\min\mbox{-}\vartriangleleft$.
\begin{Pro}
Over an arbitrary field of characteristic different to two,
$$\min\mbox{-}\vartriangleleft\subsetneq\left(\min\mbox{-}\vartriangleleft{E}\mathfrak{U}\right)^{Q}.$$
\end{Pro}

\begin{proof}
Clearly, $\min\mbox{-}\vartriangleleft$ is  a subset of
$\left(\min\mbox{-}\vartriangleleft{E}\mathfrak{U}\right)^{Q}.$ So it
remains to show that this inclusion is strict. Let  $\LieH$ be a simple
non-abelian Leibniz algebra over an arbitrary field of characteristic different to two and
let  $\{\LieH_{i}: i\geqslant 1\} $ be an infinite family of isomorphic
copy of $\LieH.$ Set $\Lieg=\oplus_{i=1}^{\infty}\LieH_{i},$ then $\Lieg$
is an infinite dimensional Leibniz algebra. Considering the
decreasing sequence, $\{J_{s}= \oplus_{i=1}^{s}\LieH_{i}, s\geqslant
1\}, $ of ideals of $\Lieg,$ one concludes that $\Lieg \notin
\min\mbox{-}\lhd$. Next, we show that $\Lieg\in
\left(\min\mbox{-}\vartriangleleft{E}\mathfrak{U}\right)^{Q}.$ Indeed, set $\LieI:= \oplus_{i\in A}\LieH_{i},$ with $A$ a
nonempty subset of $\{i\in \mathbb{N}: i\geqslant 1\}.$ $\LieI$ is
an ideal of $\Lieg$ and $\Lieg/\LieI \cong \underset{i \in
\mathbb{N}-A}{\oplus}\LieH_{i}$. So $\Lieg/\LieI \in \min\mbox{-}\lhd E
\mathfrak{U}$ and $\Lieg \in (\min\mbox{-}\lhd E \mathfrak{U})^{Q}$.
\end{proof}

\begin{Le}
    $\left(\min\mbox{-}\vartriangleleft E\mathfrak{U}\right)^{Q}=\left(\min\mbox{-}\vartriangleleft\mathfrak{U}\right)^{Q}$.
\end{Le}

\begin{proof}
Observe that  $\left(\min\mbox{-}\vartriangleleft
E\mathfrak{U}\right)^{Q} \subset
\left(\min\mbox{-}\vartriangleleft\mathfrak{U}\right)^{Q}$ since
$\min\mbox{-}\vartriangleleft E\mathfrak{U} \subset
\min\mbox{-}\vartriangleleft\mathfrak{U}$ by Remark \ref{22}. For the reverse inclusion, It is enough  to prove that  $(\min\mbox{-}\lhd \mathfrak{U})^{Q} \subseteq (\min\mbox{-}\lhd
\mathfrak{U}^{s})^{Q}$ for all $s
\geqslant 1$.  We use induction on $s$. Obviously,
this inclusion holds for $s=1$. Assume that
it holds for $s$  and let $\Lieg \in
(\min\mbox{-}\lhd \mathfrak{U})^{Q}.$ Consider
    $I_{0} \supseteq I_{1} \supseteq I_{2} \supseteq \cdots$
to be a decreasing chain of $\mathfrak{U}^{s+1}$ ideals of $\Lieg$.
For all $m \in \mathbb{N}$, $I_{m}^{(s)}$ is an ideal of $\Lieg$. In
addition, since $I_{m} \in \mathfrak{U}^{s+1}$, then $I_{m}^{(s)}
\in \mathfrak{U}$. Hence
    $I_{0}^{(s)} \supseteq I_{1}^{(s)} \supseteq I_{2}^{(s)} \supseteq \cdots$
is a descending chain of abelian ideals of $\Lieg$. Since $\Lieg \in
(\min\mbox{-}\lhd \mathfrak{U})^{Q}$, then $\Lieg \in \min\mbox{-}\lhd
\mathfrak{U},$ and thus this chain terminates. Hence there exists
$r_{0} \in \mathbb{N}$ such that
    $I_{m}^{(s)}=I_{r_{0}}^{(s)}$ for all $m \geqslant r_{0}.$
Set $I=I_{r_{0}}^{(s)}$. Then
    $\frac{I_{r_{0}}}{I} \supseteq \frac{I_{r_{0}+1}}{I} \supseteq \frac{I_{r_{0}+2}}{I} \supseteq \cdots$
is a descending chain of $\mathfrak{U}^{s}$ ideals of
$\frac{\Lieg}{I}$. By induction $\Lieg \in (\min\mbox{-}\lhd
\mathfrak{U})^{Q} \subseteq (\min\mbox{-}\lhd \mathfrak{U}^{s})^{Q}$, so
the chain terminates. Hence there exists $n_{0} \in \mathbb{N}$ such
that
    $\frac{I_{n}}{I}=\frac{I_{n_{0}}}{I}$ for all $n \geqslant n_{0}.$
So $I_{n}=I_{n_{0}}$, for all $n \geqslant n_{0}$. Hence the chain
    $I_{0} \supseteq I_{1} \supseteq I_{2} \supseteq \cdots \supseteq I_{r_{0}} \supseteq I_{r_{0}+1} \supseteq \cdots \supseteq I_{n}=I_{n_{0}}=I_{n_{0}+1}=\cdots$
terminates and so $(\min\mbox{-}\lhd \mathfrak{U})^{Q}=(\min\mbox{-}\lhd
\mathfrak{U}^{s})^{Q}$ for all $s \in \mathbb{N}$.
\end{proof}\\

\begin{Th}
 A quasi-Artinian Leibniz algebra is  Artinian if and only if every quotient algebra satisfies
 the minimal condition for abelian ideals.\\ This is symbolically equivalent to $$\min\mbox{-}\lhd =(\min\mbox{-}\lhd \mathfrak{U})^{Q} \cap (q \min\mbox{-}\lhd ).$$
\end{Th}

\begin{proof}
Consider $\Lieg \in \min\mbox{-}\lhd$ and notice that  $\Lieg \in
(\min\mbox{-}\lhd\mathfrak{U})\cap (q \min\mbox{-}\lhd ).$ Consider
$$
    \frac{J_{0}}{J} \supseteq \frac{J_{1}}{J} \supseteq \frac{J_{2}}{J} \cdots
$$
 a descending chain of abelian ideals of $\frac{\Lieg}{J}$ with $J$ being an arbitrary ideal of $\Lieg.$
This descending chain induces the following descending chain of
ideals of $\Lieg$
$$
    J_{0} \supseteq J_{1} \supseteq J_{2} \supseteq \cdots.
$$
 Since $\Lieg \in \min\mbox{-}\lhd $, then the chain
$(2)$ terminates. Hence, there exists $m_{0} \in \mathbb{N}$ such
that $J_{m}=J_{m_{0}}$, for all $m \geqslant m_{0}$. So,
$\frac{J_{m}}{J}=\frac{J_{m_{0}}}{J}$ for all $m \geqslant m_{0}$.
Hence the chain $(1)$ terminates and so $\frac{\Lieg}{J} \in \min\mbox{-}\lhd \mathfrak{U}$.
Therefore $\Lieg \in (\min\mbox{-}\lhd \mathfrak{U})^{Q}$. Hence $\Lieg \in (\min\mbox{-}\lhd \mathfrak{U})^{Q} \cap (q \min\mbox{-}\lhd )$.\\
Conversely, let $\Lieg \in (\min\mbox{-}\lhd \mathfrak{U})^{Q} \cap (q \min\mbox{-}\lhd)$ and set $I=\Lieg^{(w)}$ where
 $w$ the smallest limit ordinal. For some $d<w$, we have $I=\Lieg^{(d)}$. Consider the descending chain
$$
    J_{0} \supseteq J_{1} \supseteq J_{2} \supseteq \cdots
$$
of ideals of $\Lieg$. Then
    $I\cap J_{0} \supseteq I\cap J_{1} \supseteq I\cap J_{2} \supseteq \cdots$
is a descending chain of ideals of $\Lieg$.
 As $\Lieg \in q \min\mbox{-}\lhd $, there exists $m \in \mathbb{N}$ such that $[I,I\cap J_{m}] \subseteq \bigcap_{i\geq 0}(J_{i} \cap I)$ and
 $[I \cap J_{m},I] \subseteq \bigcap_{i\geq 0}(J_{i} \cap I)$.  Setting $J:= \bigcap_{i\geq 0}(J_{i} \cap I),$
we have
 $(J_{m+j} \cap I)^{(1)}\subseteq [I,I \cap J_{m}] \subseteq J$ for all $j\geqslant 0.$
  Consider the subchain of ideals of $(3),$
\begin{center}
    $J_{m} \supseteq J_{m+1} \supseteq J_{m+2} \supseteq \cdots$.
\end{center}
So,  for all $i \geqslant 1$, we have
$\frac{J_{m}\cap I}{J} \supseteq \frac{J_{m+1}\cap I}{J} \supseteq \frac{J_{m+2}\cap I}{J} \supseteq \cdots$
a descending chain of abelian ideals of $\Lieg/J.$ Since $\Lieg/J \in \min\mbox{-}\lhd
 E\mathfrak{U}.$
  Then this chain of abelian ideals terminates. That is, there exits $m_{0}$ such that $\frac{{J_m}\cap I}{J}=\frac{J_{m_{0}}\cap I}{J}$,
   for all $ m \geqslant m_{0} \geqslant i$.
  Hence $J_{m}\cap I=J_{m_{0}}\cap I$ for all $m \geqslant m_{0}.$ Observe on the other hand that $\frac{\Lieg}{I}$ is solvable,
  so the ideals of $\Lieg$ of the following descending chain \begin{center}
    $\frac{J_{m}+ I}{I} \supseteq \frac{J_{m+1}+ I}{I} \supseteq \frac{J_{m+2}+ I}{I} \supseteq \cdots$
\end{center}  are
  solvable and since $\Lieg \in (\min\mbox{-}\lhd \mathfrak{U})^{Q},$ this
  chain also terminates. That is, there exits $m_{1}$ such that $J_{m_{1}}+ I= J_{m_{1}+ r}+
  I$ for all $r\geq 0.$ Setting $p= max\{m_0, m_1\},$ we obtain that
  the descending chain $(3)$ terminates also. Thus $\Lieg \in \min\mbox{-}\lhd$.
\end{proof}

The following statement is an immediate consequence of the above
result.

\begin{Co}
  Let $\mathfrak{X}$ be a class of Leibniz algebras. If $\mathfrak{X}$ is $Q$-closed and $\mathfrak{X} \subseteq \min\mbox{-}\lhd \mathfrak{U}$, then $(q \min\mbox{-}\lhd \cap \mathfrak{X}) \subseteq \min\mbox{-}\lhd$.
\end{Co}

\section{ Chain Condition(CC) on ideals of ideals of Leibniz algebra.}
In this section, we investigate infinite dimensional Leibniz
algebras with chain condition(CC) on ideals of ideals.\\
\begin{De}
A Leibniz algebra $\Lieg$ is said to be semi-simple if
$Rad(\Lieg)=Leib(\Lieg).$ In particular, a semisimple Lie algebra is
a Lie algebra that does not contain a nonzero solvable ideal.
\end{De}

\begin{Pro}
Let $\Lieg$ be an infinite dimensional Leibniz algebra together with
the ascending or descending chain condition on ideals of ideals. Then every
solvable ideal $\LieK$ of $\Lieg$ is finitely generated.
\end{Pro}
\begin{proof}
The proof of the case of Leibniz algebras satisfying the decreasing
chain condition is similar to the case of Lie
algebras done in \cite{OSBOR} . Consequently,  we focus our attention here on the ascending case and
proceed by induction on the solvability index
of $\LieK.$\\
Assume that $\LieK$ abelian, that is $[\LieK, \LieK]= 0.$ Consider
the sequence, $\{\LieK_{i}: i\geq 0\},$ of ideals of $\LieK$
constructed inductively as follows: $\LieK_{0}
=\langle\{e_{0}\}\rangle$ is the ideal of $\LieK$ spanned by $e_0,$ a
nonzero vector in $\LieK;$ $\LieK_{1} =\langle\LieK_{0}\cup
\{e_{1}\}\rangle,$ with  $e_1$ being a nonzero vector in $\LieK$ such that
$\{e_0, e_1\}$ is linearly independent. Suppose that $\LieK_{j}$ is
constructed as an ideal of $\LieK$ with basis $B_{j}=\{e_{i}\in
\LieK: 0\leq i\leq j\}$ and $\LieK_{j+1}= \langle\LieK_{j}\cup
\{e_{j+1}\}\rangle.$ We have $\LieK_{0}\subseteq
\LieK_{1}\subseteq\cdots \subseteq \LieK_{j}\subseteq
\LieK_{j+1}\subseteq\cdots \subseteq \LieK.$ Since $\Lieg$ satisfies
the maximal chain condition on ideal of ideals, then the above chain
terminates, and thus $\LieK$ is of finite dimension. Suppose that the
result is established for all solvability indexes $\leq s$ and let
us show it when this index is $s+1.$ We have
$\LieK^{(s+1)}=[\LieK^{(s)}, \LieK^{(s)}]= 0,$ and so $\LieK^{(s)}$ is a
solvable ideal.  From the previous lines, it follows that
$\mbox{dim}(\LieK^{(s)})< \infty.$ So $\frac{\LieK}{\LieK^{(s)}}$ is
a solvable ideal of the quotient Leibniz algebra
$\frac{\Lieg}{\LieK^{(s)}},$ with solvability index $s.$ In addition
$\frac{\Lieg}{\LieK^{(s)}}$ also verifies the maximal chain
condition, so by the inductive
hypothesis, $\mbox{dim}(\frac{\LieK}{\LieK^{(s)}})< \infty,$
Therefore $\mbox{dim}(\LieK)< \infty.$
\end{proof}
\begin{Rem}

\begin{enumerate}
\item
If $\Lieg$ is a semisimple Leibniz algebra, then $\Lieg\cong \A
\oplus Leib(\Lieg),$ where $\A$ is a semisimple Lie algebra and
$Leib(\Lieg)$ is an $\A$-module.

\item If $\Lieg$ is an infinite dimensional Leibniz algebra over an
arbitrary field of characteristic different to two, together with the ascending or descending chain
condition, then $\Lieg$ has a unique maximal solvable ideal $\LieK$ and the
quotient algebra $\frac{\Lieg}{\LieK}$ is an infinite dimensional
semisimple Lie algebra.
\end{enumerate}
\end{Rem}

\begin{De}
Let $\Lieg$ be a semisimple Leibniz algebra. The socle $S$ of
$\Lieg$ is defined as follows:  $S:= S_1 + S_2,$ where $S_1$ is a
sum of  simple ideals of $ \A$ and $S_2$ is a sum of simple
$\A$-submodules of $Leib(\Lieg)$.
\end{De}
\begin{Pro}
 Let $\Lieg$ be a semisimple Leibniz algebra
satisfying the descending chain condition. Then  $\Lieg$ is a direct
sum of a finite direct sum of simple ideals of $\A$ and  a direct
sum of simple
 $\A$-submodule of $Leib(\Lieg),$ i.e. there exists $\{S_i: 1\leq
i\leq n\} $ a finite set of simple ideals of $\A$ such that $\Lieg=
(\bigoplus_{i=1}^{i=n}S_{i}) \oplus S_2$
 with $S_2= Leib(\Lieg)$ direct sum of simple $\A$-submodules of $Leib(\Lieg)$.
\end{Pro}
\begin{proof}
Since $\Lieg$ is a semisimple Leibniz algebra, then  its associated
Lie algebra $\A$ is also semisimple. By [\cite{OSBOR}, Lemma 4],
there exists a finite set of  simple ideals of $\A$ such that $\A=
(\bigoplus_{i=1}^{i=n}S_{i})$ and the result follows.
\end{proof}

Using the notions of prime ideal and prime algebra defined in \cite{D_{1}}, we
investigate below the structure of Leibniz algebras with a certain
chain condition(CC) on subideals.

\begin{De}
\begin{enumerate}
\item [$(i)$] A Leibniz algebra $\Lieg$ is said to be prime if $[h_{1},h_{2}] \subseteq Leib(\Lieg)$ for $h_{1},h_{2}$ ideals of $\Lieg$, then $h_{1} \subseteq Leib(\Lieg)$ or $h_{2} \subseteq Leib(\Lieg)$.
\item [$(ii)$] A proper ideal $K$ of $\Lieg$ is said to be prime if $[h_{1},h_{2}] \subseteq K$ for $h_{1},h_{2}$ ideals of $\Lieg$, then $h_{1} \subseteq K$ or $h_{2} \subseteq K$.
\item [$(iii)$] A Leibniz algebra $\Lieg$ is said to be  semiprime if $[h,h] \subseteq Leib(\Lieg)$ for an ideal $h$ of $\Lieg$, then $h \subseteq Leib(\Lieg)$.
\item [$(iv)$] A proper ideal $I$ of $\Lieg$ is said to be semiprime if $[h,h] \subseteq I$ for an ideal $h$ of $\Lieg$, then $h \subseteq I$.
\item [$(v)$] An ideal $J$ of $\Lieg$ is said to be maximal if $J \subseteq I \subseteq \Lieg$ for some ideal I of $\Lieg$, then $J=I$ or $I=\Lieg$.
\item [$(vi)$]  Let $P$ and  $I$ be two ideals of $\Lieg$ such that $P$ is prime. P is said to be a minimal prime ideal belonging
to I, if $P\supset I$ and there is no prime ideal $P_{1}$ of $\Lieg$
such that $I \subset P_{1}\subsetneq P.$
\item [$(vii)$] A prime radical of an ideal $\LieH$ of $\Lieg$, denoted $Rad_{P}(\LieH),$ is the intersection of all minimal prime
ideals of $\Lieg$ belonging to $\LieH$. In particular the prime
radical of the zero ideal of $\Lieg$ will be denoted $Rad_{P}(0):=
Rad_{P}(\Lieg)$.

\end{enumerate}
\end{De}

\begin{Rem}
Let $\Lieg$ be a Leibniz algebra. One has the following assertions:
\begin{enumerate}
\item [$(i)$] $Rad_{P}(\LieH)$ is the intersection of all prime ideals of $\Lieg$ containing $\LieH$ (see \cite{Aldosray}, Theorem 2.1.5).
\item [$(ii)$] For any two ideals $\LieH, \LieK$ of  $\Lieg,$ $Rad_{P}(\LieH \cap \LieK)= Rad_{P}(\LieH) \cap Rad_{P}(\LieK)$.\\
In particular, if  $\LieH \subseteq \LieK \subseteq
Rad_{P}(\LieH),$ then $Rad_{P}(\LieH)= Rad_{P}(\LieK).$
\item [$(iii)$] For any ideal $\LieH,$ of
 $\Lieg,$ $Rad_{P}(Rad_{P}(\LieH)= Rad_{P}(\LieH).$
 \item[$(iv)$] $Rad_{P}(\frac{\Lieg}{Rad_{P}(\Lieg)})= (0),$ the zero
 ideal of $\frac{\Lieg}{Rad_{P}(\Lieg)}.$
\end {enumerate}
\end{Rem}
\begin{Pro}

Let $\Lieg$ be a Leibniz algebra satisfying the maximal chain
condition on ideals. For any ideal $\LieH$ of $\Lieg,$ there exists
only a finite number of prime ideals of $\Lieg,$ $\{P_{i}: 1\leq
i\leq m\},$ such that $Rad_{P}(\LieH)= \bigcap_{i=1}^{i=m} P_{i}.$

\end{Pro}

\begin{proof}
Let $\Lieg$ be a Leibniz algebra  satisfying the maximal chain
condition on ideals and $\LieH$ be an arbitrary ideal of $\Lieg$. If
$\LieH$ is a prime ideal, the result follows.\\ Assume that $\LieH$
is not a prime ideal and that there is an infinite number of minimal
prime ideals of $\Lieg$ containing  $\LieH.$ Since $\LieH$ is non
prime, there exists two ideals $I, J$ of $\Lieg$, such that $[I,
J]\subseteq \LieH$
 with $I\nsubseteq \LieH$ and $J\nsubseteq \LieH.$ Observe that $[I + \LieH,
J+ \LieH]\subseteq \LieH$  and that  neither $I + \LieH$ nor $J+
\LieH$ is prime. Thus there is an infinite number of minimal prime
ideals of $\Lieg$ belonging at least to one of the following ideals
$I + \LieH$, $J+ \LieH.$ Assume that there is an infinite number of
minimal prime ideals of $\Lieg$ belonging to $I + \LieH,$ by using
the same arguments we obtain an ideal $\LieK$ such that   $\LieH
\subsetneq I + \LieH \subsetneq \LieK + I + \LieH.$ So  progressing
in the similar way, we show that there is an ascendant chain of
ideals, $\LieH \subsetneq I + \LieH \subsetneq \LieK + I + \LieH
\subsetneq \cdots,$ which does not terminates. This contradicts the
fact that $\Lieg$ satisfies the maximal chain condition.
\end{proof}
\begin{Pro}
If $\Lieg$ is a semisimple Leibniz algebra  satisfying the maximal
condition on ideals, then $Rad_{P}(\Lieg)= \mbox{Leib}(\Lieg)$.
\end{Pro}

\begin{proof}
Let $\Lieg$ be a semisimple Leibniz algebra  satisfying the maximal
condition on ideals and consider $\LieI,$ a solvable ideal of
$\Lieg$ with solvability index  $n.$ We have  $\LieI^{(n)}=0$ and
therefore $\LieI\subseteq P$ for any $P,$ minimal prime ideal of
$\Lieg$. Hence the solvable Radical, $Rad(\Lieg)$,
 is an ideal of $Rad_{P}(\Lieg).$\\
 Conversely assume that $Rad_{P}(\Lieg)$ is not solvable and consider
 $\mathfrak{M}=\{\LieK \lhd \Lieg: Rad_{P}^{(n)}(\Lieg)\nsubseteq \LieK, \forall n\in \mathbb{N} \}.$
$\mathfrak{M}\neq \emptyset$ since the zero ideal belong to
$\mathfrak{M}.$ Moreover $\mathfrak{M}$ has a maximal element since
$\Lieg$ satisfies the maximal condition on ideals. Let $Q$ be such
element and assume that $Q$ is not prime , then there exists $A, B$
two ideals of $\Lieg$ such that $[A, B]\subseteq Q$ with
$A\nsubseteq Q$ and $B\nsubseteq Q.$ But $Q + A \supseteq Q,$  $Q +
B \supseteq Q$ and $Q$ maximal implies that  there exists $r$ and
$s$ such that $Rad_{P}^{(r)}(\Lieg)\subseteq Q + A$ and
$Rad_{P}^{(s)}(\Lieg)\subseteq Q + B.$ Set $u= \mbox{max}\{r, s\} $,
then we have $Rad_{P}^{(u+1)}(\Lieg)=[Rad_{P}^{(u)}(\Lieg),
Rad_{P}^{(u)}(\Lieg)]\subseteq [Q + A, Q + B]\subseteq Q.$ This
contradicts the fact that $Q\in \mathfrak{M}.$ Hence $Q$ is prime
and $Rad_{P}^{(n)}(\Lieg)\nsubseteq Q, \forall n\in \mathbb{N}$
implies that $Rad_{P}(\Lieg)\nsubseteq Q$ which contradicts the
definition of $Rad_{P}(\Lieg).$ So $Rad_{P}(\Lieg)\subseteq Rad(\Lieg).$\\
Finally since $\Lieg$ is semi-simple we have $Leib(\Lieg)=
Rad(\Lieg)=Rad_{P}(\Lieg).$
\end{proof}


{\bf Acknowledgements}\\ The first author is grateful to the
hospitality of Max-Planck-Institut f$\ddot{u}$r Mathematik in Bonn,
where this paper was completed.
\bigskip

\end{document}